\documentclass[reqno,oneside,11pt]{amsart}
\usepackage[utf8]{inputenc} %test
\usepackage{amssymb}
\usepackage{mathrsfs}
\usepackage{graphicx}
\usepackage{tikz-cd}
\usepackage{latexsym}
\usepackage{stmaryrd}
\usepackage{enumitem}
\usepackage[bookmarks,hyperfootnotes=false, psdextra=true]{hyperref}
\usepackage{multirow}
\usepackage{xcolor}
\usepackage{caption}
\colorlet{darkishRed}{red!80!black}
\colorlet{darkishBlue}{blue!60!black}
\colorlet{darkishGreen}{green!60!black}
\hypersetup{
    draft = false,
    bookmarksopen=true,
    colorlinks,
    linkcolor={red!60!black},
    citecolor={green!60!black},
    urlcolor={blue!60!black}
}
\usepackage[nameinlink, capitalise, noabbrev]{cleveref}
\usepackage{thmtools}
\usepackage{nameref}
\crefformat{enumi}{#2#1#3}
\crefformat{equation}{#2(#1)#3}
\usepackage[msc-links]{amsrefs} 
\usepackage{doi}

\renewcommand{\PrintDOI}[1]{\doi{#1}}

\let\sm=\setminus
\usepackage{relsize}
\usepackage{comment}
\usepackage{svg}
\usepackage{mathtools}
\usepackage{nccmath}

\usepackage{geometry}
\renewcommand{\leq}{\leqslant}
\renewcommand{\geq}{\geqslant}
\renewcommand{\ge}{\geq}
\renewcommand{\le}{\leq}

\newcommand{\rminor}{\succcurlyeq}

\usepackage{soul}

\let\rho=\varrho
\let\phi=\varphi

\renewcommand{\subset}{\subseteq}

\newcommand{\se}{\subset}
\newcommand{ \N } { \mathbb{N} }

\makeatletter

\def\calCommandfactory#1{%
   \expandafter\def\csname c#1\endcsname{\mathcal{#1}}}
\def\frakCommandfactory#1{%
   \expandafter\def\csname frak#1\endcsname{\mathfrak{#1}}}

\newcounter{ctr}
\loop
  \stepcounter{ctr}
  \edef\X{\@Alph\c@ctr}
  \expandafter\calCommandfactory\X
  \expandafter\frakCommandfactory\X
  \edef\Y{\@alph\c@ctr}
  \expandafter\frakCommandfactory\Y
\ifnum\thectr<26
\repeat

\renewcommand{\cC}{\mathscr{C}}
\setenumerate{label={\normalfont (\roman*)}}%,itemsep=0pt}

 \def\lowfwd #1#2#3{{\mathop{\kern0pt #1}\limits^{\kern#2pt\raise.#3ex
 \vbox to 0pt{\hbox{$\scriptscriptstyle\rightarrow$}\vss}}}}
 \def\lowbkwd #1#2#3{{\mathop{\kern0pt #1}\limits^{\kern#2pt\raise.#3ex
 \vbox to 0pt{\hbox{$\scriptscriptstyle\leftarrow$}\vss}}}}
 \def\fwd #1#2{{\lowfwd{#1}{#2}{15}}}
 \def\vS{{\vec S}_{\aleph_0}}

 \def\vS{{\hskip-1pt{\fwd S3}\hskip-1pt}}

 \def\ve{\kern-1.5pt\lowfwd e{1.5}2\kern-1pt}
 \def\ev{\kern-1pt\lowbkwd e{0.5}2\kern-1pt}
 \def\vf{\kern-2pt\lowfwd f{2.5}2\kern-1pt}
 \def\vsdash{{\mathop{\kern0pt s\lower.5pt\hbox{${}% logically \v(e')
     \scriptstyle'$}}\limits^{\kern0pt\raise.02ex
     \vbox to 0pt{\hbox{$\scriptscriptstyle\rightarrow$}\vss}}}}
 \def\svdash{{\mathop{\kern0pt s\lower.5pt\hbox{${}% logically \v(e')
     \scriptstyle'$}}\limits^{\kern0pt\raise.02ex
     \vbox to 0pt{\hbox{$\scriptscriptstyle\leftarrow$}\vss}}}}

 \def\vrdash{{\mathop{\kern0pt r\lower.5pt\hbox{${}% logically \v(e')
     \scriptstyle'$}}\limits^{\kern0pt\raise.02ex
     \vbox to 0pt{\hbox{$\scriptscriptstyle\rightarrow$}\vss}}}}
 \def\rvdash{{\mathop{\kern0pt r\lower.5pt\hbox{${}% logically \v(e')
     \scriptstyle'$}}\limits^{\kern0pt\raise.02ex
     \vbox to 0pt{\hbox{$\scriptscriptstyle\leftarrow$}\vss}}}}

 \def\vSd{{\mathop{\kern0pt S\lower-1pt\hbox{${}% logically \v(e')
      \scriptstyle'$}}\limits^{\kern2pt\raise.1ex
      \vbox to 0pt{\hbox{$\scriptscriptstyle\rightarrow$}\vss}}}}
\newtheorem{theorem}{Theorem}[section] 
\newtheorem{proposition}[theorem]{Proposition}%[section]
\newtheorem{corollary}[theorem]{Corollary}
\newtheorem{lemma}[theorem]{Lemma}

\newtheorem{mainresult}{Theorem} 

\newtheorem{mainconjecture}[mainresult]{Conjecture}

\newtheorem{claim}{Claim}[theorem]

\newenvironment{cproof}{\paragraph{\textit{Proof of Claim.}}}{\phantom{!}\!\hfill$\diamondsuit$}

\theoremstyle{definition}
\newtheorem{example}[theorem]{Example}

\theoremstyle{remark}

\newtheorem*{ack}{Acknowledgement}

\usepackage{geometry}

\title[Characterising 4-tangles through a connectivity property]{Characterising 4-tangles through\\a connectivity property}
\author{Johannes Carmesin${}^\clubsuit$}
\author{Jan Kurkofka${}^\clubsuit$}
\thanks{${}^\clubsuit$University of Birmingham, Birmingham, UK, funded by EPSRC, grant number EP/T016221/1}

\begin{document}
\vspace*{-1.0cm}
\begin{abstract}
Every large $k$-connected graph-minor induces a $k$-tangle in its ambient graph.
The converse holds for $k\le 3$, but fails for $k\ge 4$.
This raises the question whether `$k$-connected' can be relaxed to obtain a characterisation of $k$-tangles through highly cohesive graph-minors.
We show that this can be achieved for $k=4$ by proving that internally 4-connected graphs have unique 4-tangles, and that every graph with a 4-tangle $\tau$ has an internally 4-connected minor whose unique 4-tangle lifts to~$\tau$.
\end{abstract}
\maketitle

\vspace*{-.5cm}
\section{Introduction}

Tangles have been introduced by Robertson and Seymour~\cite{GMX} as an abstract notion that unifies many established concrete notions of what a highly cohesive substructure somewhere in a graph could be.
Curiously, the converse is open: there is no established concrete notion for high cohesion such that every tangle is induced by a highly cohesive substructure.

It is well-known that 2-connectivity characterises 2-tangles, and 3-connectivity characterises 3-tangles, in the following sense.
Recall that a \emph{graph-property} is a class of graphs closed under isomorphism.
A graph-property \emph{characterises} $k$-tangles for a $k\in\N$ if, on the one hand, every graph with the property has a unique $k$-tangle, and on the other hand, for every $k$-tangle $\tau$ in a graph $G$ there is a graph-minor $H$ of $G$ which exhibits the property and such that the unique $k$-tangle in~$H$ lifts to~$\tau$; see \cref{sec:TnT} for the definition of \emph{lifts}.
Graphs with the cube as a contraction-minor show that 4-connectivity fails to characterise 4-tangles, and more generally $k$-connectivity fails to characterise $k$-tangles for $k\ge 4$.

As our main result, we show that 4-tangles are characterised by a relaxation of 4-connectivity known as internal 4-connectivity.
A~graph $G$ is \emph{internally 4-connected} if it is 3-connected, has more than four vertices, and every 3-separator of $G$ is independent and separates only one vertex from the rest of~$G$.

\begin{mainresult}\label{main}
Internal 4-connectedness characterises 4-tangles.
\end{mainresult}

\begin{mainconjecture}
For every $k\in\N$ there is a connectivity-property that characterises $k$-tangles.
\end{mainconjecture}

The key ingredient for the proof of \cref{main} is a new decomposition theorem that provides a tree-decomposition of every 3-connected graph into torsos that are internally 4-connected or $K_4$ or~$K_3$, and such that distinct 4-tangles live in distinct bags; see \cref{greedyTangleProperties}.
We also show how internal 4-connectedness can be used to prove Kuratowski's theorem.

\subsection*{Related work.}

Over the past years, a rich Theory of Tangles has been established 
\cites{klepper2022clustering,GraphDecomp,albrechtsen2023optimal,albrechtsen2023refining,CarmesinToTshort,CDHH13CanonicalAlg,CDHH13CanonicalParts,confing,CG14:isolatingblocks,Entanglements,ASS,TreeSets,StructuralASS,ProfilesNew,FiniteSplinters,InfiniteSplinters,RefiningToT,GroheTangles3,ComputingTangles,jacobs2023efficiently,GMX,GGW:tangles_in_matroids,ReedConnectivityMeasure,diestel_ultrafilter-tangles,diestel2016tangles,diestel2019tangles,diestel2020tangle}.
Grohe has studied 4-tangles as well~\cite{grohe2016quasi}.
He showed that quasi 4-connectivity\footnote{A graph is \emph{quasi 4-connected} if it is 3-connected, has more than four vertices, and every 3-separation of~$G$ has a side of size~$\le 4$.} characterises 4-tangles, but only up to two counterexamples that can occur in infinitely many graphs \cite{grohe2016quasi}*{Theorem~4.5}.
This inspired us to search for another relaxation of 4-connectivity that might work without any restrictions whatsoever.
Internal 4-connectivity is an established relaxation of 4-connectivity, lying in between quasi 4-connectivity and 4-connectivity, which achieves just that.

Our decomposition result \cref{greedyTangleProperties} is similar to Grohe's~\cite{grohe2016quasi}.
There are two differences worth pointing out.
Grohe shows that the torsos he obtains are quasi 4-connected or $K_4$ or $K_3$, while we obtain internally 4-connected torsos or $K_4$ or $K_3$, and internally 4-connected graphs are quasi 4-connected.
The other difference is that our decomposition arises from a short greedy construction\footnote{A 3-separation $\{A,B\}$ of $G$ is \emph{claw-free} if neither $G[A]$ nor $G[B]$ is a claw aka~$K_{1,3}$. A set $S$ of proper 3-separations is \emph{claw-freeable} if $S$ can be enumerated as $s_0,s_1,\ldots,s_n$ such that for every index~$i$ either $s_i$ is claw-free or there is $j<i$ such that the separators of $s_j$ and $s_i$ intersect in at least two vertices. 
Take the tree-decomposition defined by an inclusionwise maximal claw-freeable nested set of 3-separations -- that's it.},
while Grohe's exciting construction is far more involved.

We unequivocally stress that \cref{greedyTangleProperties} is much weaker than the tri-separation decomposition from~\cite{Canonical3arXiv}.
The decompositions of \cref{greedyTangleProperties} can be obtained as refinements of the tri-separation decomposition with some effort, but it is not possible to obtain the tri-separation decomposition from \cref{greedyTangleProperties}.
This is because the tri-separation decomposition is \emph{canonical}, meaning that it is invariant under graph-automorphisms, and because it admits an \emph{explicit} description that uniquely determines it for every graph.
The tri-separation decomposition has applications for which \cref{greedyTangleProperties} is not suited because the decompositions provided by \cref{greedyTangleProperties} are neither canonical nor explicit.
These applications include Connectivity Augmentation, Cayley graphs or Parallel Computing.

We introduce \cref{greedyTangleProperties} as a handy tool for those situations where canonicity or an explicit description are not crucial.
Of course, we could also use~\cite{grohe2016quasi} instead of \cref{greedyTangleProperties}, but we like the alternative construction we have for \cite{grohe2016quasi} and think it is worth sharing.
We would also like to mention that an extension to 4-connectivity of the tri-separation decomposition~\cite{Canonical3arXiv} has recently been found~\cite{Tutte4con}.

For matroids, a version of \cref{main} has been proved in parallel and independently from our project, by Brettell, Jowett, Oxley, Semple and Whittle~\cite{4connectedMatroids}.
They use weak 4-connectivity, which is not as strong as internal 4-connectivity but optimal in the setting of matroids.
There are differences in our approaches and the results do not imply each other in an obvious way.
\medskip

This paper is organised as follows.
\cref{sec:TnT} investigates two substructures in graphs that induce 4-tangles.
In~\cref{sec:GreedyDecomp}, we prove the decomposition result \cref{greedyTangleProperties}.
This is used at the end of the section to deduce \cref{main}.
In \cref{sec:Universality}, we show that the relevant torsos of the tree-decompositions provided by \cref{greedyTangleProperties} are unique up to isomorphism.
In \cref{sec:Kuratowski}, we provide a proof of Kuratowski's theorem that uses internal 4-connectedness.

\section{Tangles from minors}\label{sec:TnT}

For terminology regarding graphs, minors, separations and tangles, we follow~\cite{DiestelBookCurrent}.
A separation $\{A,B\}$ is \emph{proper} if $A\sm B\neq\emptyset\neq B\sm A$.
We refer to $K_{1,3}$ as a \emph{claw}.
%every 3-separation\footnote{A~pair $\{A,B\}$ such that $A\cup B=V(G)$ and there are no $(A\sm B)$--$(B\sm A)$ edges in~$G$ and $|A\cap B|=3$.} of~$G$ has a side such that the induced subgraph on that side is a claw (aka~$K_{1,3}$) or has only three vertices.

Suppose that $G$ and $H$ are two graphs where $H$ is a minor of~$G$.
Then there are a vertex set $U\subset V(G)$ and a surjection $f\colon U\to V(H)$ such that the preimages $f^{-1}(x)\subset U$ form the branch sets of a model of $H$ in $G$.
A \emph{minor-map} $\varphi\colon G\rminor H$ formally is such a pair $(U,f)$.
Given $\varphi=(U,f)$ we address $U$ as $V(\varphi)$ and we write $\varphi=f$ by abuse of notation.
Usually, we will abbreviate `minor-map' as `map'.

If $\varphi\colon G\rminor H$ and $s=\{A,B\}$ is a separation of~$G$, then $s$ \emph{induces} the separation $\varphi(s):=\{A_\varphi,B_\varphi\}$ of~$H$, where $A_\varphi$ consists of those vertices of $H$ whose branch set contains a vertex of~$A$, and $B_\varphi$ is defined analogously.
The order of $\varphi(s)$ is no larger than the order of~$s$.
For oriented separations~$s$, we define $\varphi(s)$ analogously.
If $\tau$ is a $k$-tangle in~$H$, then the $\varphi$-\emph{lift} of~$\tau$ to~$G$ is the collection of all oriented $({<}\,k)$-separations $s$ of~$G$ with $\varphi(s)\in\tau$.

\begin{lemma}\cite[(6.1)]{GMX}\label{LiftingTangles}
Let $\varphi\colon G\rminor H$.
Then the $\varphi$-lift to~$G$ of every $k$-tangle in~$H$ is a $k$-tangle in~$G$.
\end{lemma}
\begin{proof}
Let $\tau'$ be the $\varphi$-lift to~$G$ of a $k$-tangle~$\tau$ in~$H$.
If $\tau'$ contains oriented separations $(A_i',B_i')$ for $i\in [3]$ such that $G[A_1']\cup G[A_2']\cup G[A_3']=G$, then $H[A_1]\cup H[A_2]\cup H[A_3]=H$ for $(A_i,B_i):=\varphi(A_i',B_i')$, contradicting the fact that~$\tau$ is a tangle in~$H$.
It remains to show that $\tau'$ does not contain both orientations of the same separation $\{A,B\}$ of~$G$.
Indeed, this could only happen if $\varphi$ sends $\{A,B\}$ to $\{V(H),V(H)\}$, in which case $H$ would have at most $|A\cap B|<k$ vertices; but graphs on less than $k$ vertices have no $k$-tangles, so this cannot happen.
\end{proof}

\subsection{Internal 4-connectivity}\label{sec:Internal4con}

%Recall that a graph~$G$ is \emph{internally 4-connected} if it is 3-connected, every proper 3-separation of $G$ has a side whose induced subgraph is a claw, and $G\notin\{K_4,K_{3,3}\}$.

\begin{lemma}\label{4tangleAvoidsClaw}
Let $\{A,B\}$ be a 3-separation of graph~$G$ such that $|A|=3$ or the induced subgraph $G[A]$ is a claw.
Then every 4-tangle in~$G$ orients $\{A,B\}$ towards~$B$.
\end{lemma}
\begin{proof}
Let $\tau$ be a 4-tangle in~$G$.
If $|A|=3$, then $G[B]=G$, so $\tau$ cannot orient $\{A,B\}$ towards~$A$.
Suppose that $G[A]$ is a claw.
Then $G[A]$ can be written as the union of two paths $P_1,P_2$ with two edges.
For both $i=1,2$, the 4-tangle~$\tau$ orients $\{V(P_i),V(G)\}$ towards~$V(G)$.
As $G=G[B]\cup P_1\cup P_2$, the 4-tangle~$\tau$ must orient $\{A,B\}$ towards~$B$.
\end{proof}

\begin{proposition}\label{internalUnique}
Let $G$ be an internally 4-connected graph.
Then every 3-separation of~$G$ has one side that is larger than the other.
Orienting every 3-separation of $G$ towards its largest side defines a 4-tangle in~$G$, and this is the only 4-tangle in~$G$.
\end{proposition}
\begin{proof}
Let $G$ be an internally 4-connected graph.
Then $G$ has at least six vertices or $G=K_5$.
Hence every 3-separation of $G$ has one side that contains more vertices than the other, and we let $\tau$ orient each 3-separation to its largest side.
Suppose for a contradiction that $\tau$ contains separations $(A_1,B_1),(A_2,B_2),(A_3,B_3)$ such that $G[A_1]\cup G[A_2]\cup G[A_3]=G$.
Since $G$ is internally 4-connected and $|A_i|<|B_i|$, each $G[A_i]$ has only three vertices or is a claw, so $G$ has at most nine edges.
Combining this with the fact that $G$ has minimum degree at least three, we find that $G$ has exactly six vertices and is 3-regular.
Combining 3-regularity with internal 4-connectedness, we find that the neighbourhoods of the vertices of $G$ form independent sets.
We now derive that $G=K_{3,3}$, as follows.
Let $v$ be an arbitrary vertex of~$G$, and let $X$ denote its neighbourhood.
Since $|V(G)|=6$, the graph $G$ has exactly two more vertices $a,b$ besides the vertices in $\{v\}\cup X$.
Since $G[\{v\}\cup X]$ is a claw and $G$ is 3-regular, every vertex in $X$ is adjacent to both $a$ and~$b$.
So $G$ contains a spanning $K_{3,3}$, and by 3-regularity we have $G=K_{3,3}$.
But $G=K_{3,3}$ contradicts that $G$ is internally 4-connected.
Therefore, $\tau$ is a 4-tangle.

Suppose now for a contradiction that $\tau'$ is another 4-tangle in~$G$.
Then $\tau'$ orients some 3-separation $\{A,B\}$ of~$G$ differently than~$\tau$.
Since $G$ is internally 4-connected, one of $G[A]$ or $G[B]$ has only three vertices or is a claw, contradicting \cref{4tangleAvoidsClaw}.
Thus, $\tau$ is unique.
\end{proof}

\subsection{Cubes}\label{sec:Cubes}

The \emph{cube} means the graph that resembles the 3-dimensional cube, see \cref{fig:cubeBip}.
A set $X$ of four vertices in a graph~$G$ is \emph{cubic} in~$G$ if the set of neighbourhoods of the components of $G\sm X$ is equal to~$[X]^3$ (the collection of all subsets of $X$ of size three).

\begin{figure}[ht]
    \centering
    \includegraphics[height=4\baselineskip]{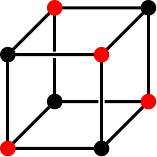}
    \caption{The red and black vertex sets are cubic in the cube}
    \label{fig:cubeBip}
\end{figure}

\begin{example}\label{cubicVXcubeMinor}
We claim that if $X\se V(G)$ is cubic in~$G$, then $G$ contains the cube as a minor. 
Indeed, let us view the cube~$Q$ as a bipartite graph with classes $A$ and~$B$.
For the branch sets of the vertices in $A$ we take the vertex sets of components $C_x$ of $G\sm X$ with neighbourhoods equal to $X-x$, for $x\in X$, while for the vertices in~$B$ we take the singletons $\{x\}$ as branch sets.
\end{example}

Given a cubic vertex set~$X$ in a graph~$G$, we can obtain a minor-map $\varphi\colon G\rminor Q$ to the cube~$Q$ as in \cref{cubicVXcubeMinor}.
We call such minor-maps \emph{standard cube-minors} at~$X$.

\begin{lemma}\label{TangleDirection}
\cite[§12 Excercise~43]{DiestelBookCurrent}
Let $\tau$ be a $k$-tangle in a graph~$G$.
Then for every set $X\se V(G)$ of fewer than $k$ vertices there exists a unique component $C=C(X,\tau)$ of $G\sm X$ such that, for every $({<}\,k)$-separation $\{A,B\}$ of $G$ with $A\cap B=X$, the $k$-tangle~$\tau$ orients $\{A,B\}$ to the side that includes the component~$C$.
\end{lemma}

Recall that a set $\{\,(A_i,B_i):i\in I\,\}$ of (oriented) separations $(A_i,B_i)$ of a graph~$G$ is a \emph{star} if $(A_i,B_i)\le (B_j,A_j)$ for all~$i\neq j\in I$.
The \emph{bag} of this star is $\bigcap_{i\in I}B_i$ (where the empty intersection is set to be~$V(G)$).

\begin{proposition}\label{cubicTanglesInFourSets}
Let $G$ be a 3-connected graph.
Let $\sigma$ be a star of 3-separations of~$G$ such that the bag $X$ of~$\sigma$ has four vertices.
Then the following assertions are equivalent:
\begin{enumerate}
    \item $\sigma$ is included in some 4-tangle in~$G$;
    \item $X$ is cubic in~$G$;
    \item $X$ is cubic in~$G$, and for every standard cube-minor $\varphi\colon G\rminor Q$ at~$X$ the $\varphi$-lift to~$G$ of the unique 4-tangle in~$Q$ includes~$\sigma$.
\end{enumerate}
\end{proposition}
\begin{proof}
(i)$\to$(ii).
Let $\tau$ be a 4-tangle in~$G$ with $\sigma\se\tau$.
Suppose for a contradiction that $X$ is not cubic in~$G$.
Then one of the four elements of $[X]^3$ is not equal to the neighbourhood of any component of~$G\sm X$.
Let $X_1,X_2,X_3$ be the other three elements of~$[X]^3$.
For each $i\in [3]$, let $U_i$ denote the (possibly empty) union of all components of $G\sm X$ with neighourhood equal to~$X_i$, and let $s_i:=\{V(U_i)\cup X_i,V(G\sm U_i)\}$.
Since all $s_i$ are 3-separations of $G$ and $\bigcup_{i=1}^3 G[U_i\cup X_i]=G$, no 4-tangle in~$G$ orients all $s_i$ towards~$G\sm U_i$.
Hence $\tau$ orients~$s_1$, say, towards~$U_1\cup X_1$.
As $\tau$ is a 4-tangle, $U_1$ must include the component $C(X_1,\tau)=:C$ of $G\sm X_1$ provided by \cref{TangleDirection}.
Since $\sigma$ is a star of 3-separations with bag~$X$, some $(A,B)\in\sigma$ satisfies $C\se G[A\sm B]$.
But then $\tau$ orients $\{A,B\}$ towards~$A$ by \cref{TangleDirection}, contradicting the assumption that $\sigma\se\tau$.

(ii)$\to$(iii).
The cubic vertex set $X$ gives rise to a standard cube-minor $\varphi\colon G\rminor Q$.
Since the cube is internally 4-connected, it has a unique 4-tangle $\tau_Q$ by \cref{internalUnique}.
This 4-tangle lifts to a 4-tangle $\tau$ in~$G$ by \cref{LiftingTangles}.
Every $s\in \sigma$ induces an oriented 3-separation~$\varphi(s)$ of the cube minor~$Q$.
The separator of~$s$ is an element $Y\in [X]^3$, and the left side of~$s$ includes only components of $G\sm X$ with neighbourhood equal to~$Y$.
Hence the left side of~$\varphi(s)$ either induces a claw in~$Q$ or has size three, depending on whether one of the components included in its left side is used as a branch set by the cube minor.
In either case, $\varphi(s)\in\tau_Q$ by \cref{4tangleAvoidsClaw}.
Hence $s\in\tau$ and, more generally,~$\sigma\se\tau$.

(iii)$\to$(i) is trivial.
\end{proof}

\section{Greedy decompositions of 3-connected graphs}\label{sec:GreedyDecomp}

The aim of this section is to find a tree-decomposition of every 3-connected graph from which we can read the internally 4-connected minors that we need for a proof of \cref{main}.
This will be achieved with \cref{greedyTangleProperties}.

\subsection{Preparation}

\begin{lemma}\label{3sepClawChar}
Every proper 3-separation $\{A,B\}$ of a 3-connected graph~$G$ satisfies exactly one of the following:
\begin{enumerate}[label={\textnormal{(C\arabic*)}}]
    \item\label{C:claw} $G[A]$ or $G[B]$ is a claw with set of leafs equal to $A\cap B$ while the other includes a cycle;
    \item\label{C:twoCycles} both $G[A]$ and $G[B]$ include cycles.
\end{enumerate}
\end{lemma}
\begin{proof}
Clearly, \cref{C:claw} and \cref{C:twoCycles} exclude each other.
Suppose that \cref{C:twoCycles} fails; we shall show~\cref{C:claw}.
Since $G$ has minimum degree three, it suffices to show for both $X\in\{A,B\}$ that $G[X]$ is a claw if $G[X]$ includes no cycle.
Say $X=A$.
Every component of $G[A\sm B]$ has neighbourhood equal to $A\cap B$, so $G[A]$ is connected.
Thus, if $G[A]$ includes no cycle, it is a tree~$T$.
All leafs of $T$ lie in $A\cap B$.
Since $A\cap B$ has size three, $T$ has at most three leafs.
As every non-leaf of $T$ has degree at least three, $T$ has at most one non-leaf.
Since $A\sm B$ is nonempty, $T$ has at least one non-leaf.
So $T$ is a claw with set of leafs equal to~$A\cap B$.
\end{proof}

A~proper 3-separation $\{A,B\}$ of $G$ is \emph{claw-free} if it satisfies \cref{C:twoCycles} instead of \cref{C:claw}.\medskip

Let $\sigma=\{\,(A_i,B_i): i\in I\,\}$ be a star of separations of~$G$.
The \emph{torso} of~$\sigma$ is the graph obtained from the subgraph of~$G$ induced by the bag of~$\sigma$ by turning each separator $A_i\cap B_i$ (for $i\in I$) into a clique.
Now let $S$ be a set of separations of~$G$ and recall that $\vS=\{\,(A,B),(B,A):\{A,B\}\in S\,\}$.
We say that $\sigma$ is a \emph{splitting star} of~$S$ if $\sigma\se\vS$ and for every $\{C,D\}\in S$ there is $i\in I$ such that either $(C,D)\le (A_i,B_i)$ or $(D,C)\le (A_i,B_i)$.
A separation $\{U,W\}$ of~$G$ \emph{interlaces} a star~$\sigma$ if for every $i\in I$ either $(A_i,B_i)<(U,W)$ or $(A_i,B_i)<(W,U)$.

\begin{lemma}\label{LiftingLemma}
Let $N$ be a nested set of separations of a graph~$G$, and let $\sigma$ be a splitting star of~$N$ with torso~$X$.
For every proper separation $\{A,B\}$ of~$X$ there exists a separation $\{\hat A,\hat B\}$ of~$G$ such that $\hat A\cap V(X)=A$ and $\hat B\cap V(X)=B$ and $\hat A\cap \hat B=A\cap B$.
Moreover, $\{\hat A,\hat B\}$ interlaces~$\sigma$ and is nested with~$N$.
\end{lemma}
\begin{proof}
This is folklore.
The proof is analogous to the proof of \cite[Lemma~2.6.4]{Canonical3arXiv}.
\end{proof}

In the context of \cref{LiftingLemma}, we say that $\{A,B\}$ \emph{lifts} to $\{\hat A,\hat B\}$, and call $\{\hat A,\hat B\}$ a \emph{lift} of $\{A,B\}$.

\subsection{Decomposition}

The purpose of the following definitions is to get \cref{chopExtension} and \cref{greedyTangleProperties} to work.

A set $S$ of proper 3-separations is \emph{claw-freeable} if $S$ can be enumerated as $s_0,s_1,\ldots,s_n$ such that for every index~$i$ either $s_i$ is claw-free or there is $j<i$ such that the separators of $s_j$ and $s_i$ intersect in at least two vertices.
In this context, the linear ordering $s_0,\ldots,s_n$ of~$S$ is called \emph{claw-freeing}.
A~claw-freeable nested set of proper 3-separations is a \emph{3-chop}.
A~3-chop of~$G$ is \emph{maximal} if it is not properly included as a subset in another 3-chop of~$G$.

\begin{figure}[ht]   
\begin{center}
   	\includegraphics[height=7\baselineskip]{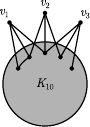}
   	\caption{This graph is obtained from a $K_{10}$ by first attaching the three vertices $v_1$, $v_2$ and $v_3$ of degree three as illustrated, and then deleting all edges with both ends in the neighbourhood of a~$v_i$, except one edge in the neighbourhood of~$v_1$ which lies as in the figure.}\label{fig:ord}
\end{center}
\end{figure}

\begin{example}
Let $G$ be the graph depicted in \cref{fig:ord}.
Let $N$ denote the set of all 3-separations of the form $s_i:=\{V(G)-v_i,\{v_i\}\cup N(v_i)\}$ where $i\in [3]$.
Then $N$ has a claw-freeing linear ordering~$s_1,s_2,s_3$.
This is the only claw-freeing ordering of~$N$.
\end{example}

Let $H$ and $G$ be two graphs.
We say that $H$ is a \emph{faithful minor} of~$G$ if there is $\varphi\colon G\rminor H$ such that $x\in\varphi^{-1}(x)$ for all $x\in V(H)$.
Note that $H$ being a faithful minor of~$G$ implies $V(H)\se V(G)$.

\begin{lemma}\label{faithfulSimpleTorso}
    Let $\{A,B\}$ be a proper 3-separation of a 3-connected graph~$G$.
    If $G[B]$ includes a cycle, then the torso of $\{(B,A)\}$ is a faithful minor of~$G$.
\end{lemma}
\begin{proof}
    The torso $H$ of $\{(B,A)\}$ is equal to the graph that is obtained from $G[A]$ by turning $A\cap B$ into a triangle.
    Let $O\se G[B]$ be a cycle.
    By Menger's theorem, we find a set $\cP$ of three disjoint paths in $G$ from $A\cap B$ to~$O$.
    We now find $H$ as a faithful minor of $G$ by deleting all the vertices and edges of $G[B]$ that are not in $O\cup \bigcup\cP$, contracting the three paths in $\cP$ to vertices, and contracting the three segments of $O$ between the endvertices of paths in~$\cP$ to edges.
\end{proof}

A nested set~$N$ of separations of a graph~$G$ is \emph{torso-faithful} to~$G$ if every torso of~$N$ is a faithful minor of~$G$.

\begin{lemma}\label{chopExtension}
All 3-chops of 3-connected graphs are torso-faithful.
\end{lemma}
\begin{proof}
Let $N$ be a 3-chop of a 3-connected graph~$G$.
We proceed by induction on~$|N|$.
If $N$ is empty, there is nothing to show.
So let $\{A,B\}$ be the least element in a claw-freeing linear odering of~$N$.
Then $\{A,B\}$ is claw-free, so $G[A]$ and $G[B]$ include cycles by \cref{3sepClawChar}.

Let $G_A$ denote the graph obtained from $G[A]$ by turning $A\cap B$ into a triangle.
Then $G_A$ is 3-connected by \cref{LiftingLemma}.
Let $N_A'$ consist of all $\{C,D\}\in N$ that satisfy either $(C,D)<(A,B)$ or $(D,C)<(A,B)$, and let $N_A$ consist of $\{C\cap A,D\cap A\}$ for all $\{C,D\}\in N_A'$.
Note that $N_A$ is a claw-freeable nested set of proper 3-separations (in~$G_A$) and hence a 3-chop.
Thus $N_A$ is torso-faithful to~$G_A$ by the induction hypothesis.
By \cref{faithfulSimpleTorso}, $G_A$ is a faithful minor of~$G$.
Hence every torso of $N_A$ in~$G_A$ is a faithful minor of~$G$.
A~symmetric definition and argumentation show that every torso of~$N_B$ is a faithful minor of~$G$.

Let $\sigma$ be an arbitrary splitting star of~$N$.
If $\sigma$ contains neither $(A,B)$ nor $(B,A)$, then $\sigma$ defines a splitting star of~$N_A$ or of~$N_B$ with the same torso as~$\sigma$, and so the torso of~$\sigma$ is a faithful minor of~$G$ as shown above.
Otherwise $\sigma$ contains $(B,A)$, say.
Then $\sigma_A:=\sigma\cap N_A'$ defines a splitting star of~$N_A$.
The torso of $\sigma_A$ in $G_A$ is equal to the torso of $\sigma$ in $G$, and so it is a faithful minor of $G$ as shown above.
\end{proof}

\begin{figure}[ht]
    \centering
    \includegraphics[height=7\baselineskip]{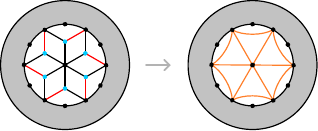}
    \caption{The grey donut is a large complete graph}
    \label{fig:greyDonut}
\end{figure}

\begin{example}\label{eg:tffNotFreeable}
The converse of \cref{chopExtension} fails in the following sense: There is a 3-connected graph~$G$ with a nested set~$N$ of proper 3-separations of~$G$ such that $N$ is torso-faithful but not claw-freeable.
\end{example}
\begin{proof}
Consider a graph $G$ as depicted in \cref{fig:greyDonut} on the left-hand side.
Let $N$ denote the set of all 3-separations whose separator is the neighbourhood of a blue vertex.
The torsos of the splitting stars of size one are faithful minors of~$G$.
The only nontrivial splitting star has the graph depicted on the right-hand side of \cref{fig:greyDonut} as torso.
This torso is also a faithful minor of~$G$, as it can be obtained from $G$ by contracting all red edges.
However, $N$ is not claw-freeable, since no element of $N$ is claw-free and the first element of a claw-freeing ordering must be claw-free.
\end{proof}

\cref{eg:tffNotFreeable} raises the question why we do not replace `claw-freeable' with `torso-faithful' in the definition of 3-chop.
In fact, we could do this and \cref{greedyTangleProperties} would hold with this alternative definition of 3-chop.
But it turns out that we need the definition via `claw-freeable' instead of `torso-faithful' to show \cref{universality}.

\begin{lemma}\label{largeTorsoInt4con}
Let $G$ be a 3-connected graph other than~$K_{3,3}$.
Let $\sigma$ be a splitting star of a maximal 3-chop~$N$ of~$G$ with torso~$X$.
If $|X|\ge 5$, then $X$ is internally 4-connected.
\end{lemma}
\begin{proof}
The torso $X$ is 3-connected, by \cref{LiftingLemma} and since $|X|\ge 5$.
We have $X\neq K_4$ since $|X|\ge 5$.
If $X$ is a $K_{3,3}$, then $\sigma$ must be empty since $X$ contains no triangle, and so $G=X=K_{3,3}$ would contradict our assumptions.

Suppose for a contradiction that $X$ has a claw-free proper 3-separation~$\{A,B\}$.
Let $\{\hat A,\hat B\}$ be a lift of $\{A,B\}$ from~$X$ to~$G$, which interlaces~$\sigma$ and is nested with~$N$, by \cref{LiftingLemma}.
Let $N'$ be obtained from $N$ by adding the lift~$\{\hat A,\hat B\}$.
If the lift $\{\hat A,\hat B\}$ is claw-free, then $N'$ is a 3-chop of~$G$ by \cref{chopExtension}, which contradicts the maximality of~$N$.

Otherwise $G[\hat A]$ is a claw, say.
Since $\{\hat A,\hat B\}$ is a lift of $\{A,B\}$ and $\{A,B\}$ is proper, we have $\hat A=A$.
Since $X[A]$ includes a cycle but $G[\hat A]\se X[A]$ is a claw, some edge in $X[A\cap B]$ must be missing in~$G$.
So there is an element of~$\sigma$ whose separator intersects $A\cap B$ in at least two vertices.
As $A\cap B=\hat A\cap\hat B$, the set $N'$ is a 3-chop of~$G$, which contradicts the maximality of~$N$.
\end{proof}

\begin{lemma}\label{liftIncludesStar}
Let $G$ be a 3-connected graph.
Let $\sigma$ be a star of 3-separations of~$G$ such that the torso~$X$ of $\sigma$ is a faithful minor of~$G$ witnessed by $\varphi\colon G\rminor X$.
Then the $\varphi$-lift to~$G$ of every 4-tangle in~$X$ includes~$\sigma$.
\end{lemma}
\begin{proof}
Otherwise some 4-tangle in~$X$ would contain $\varphi(B,A)=(B\cap V(X),A\cap B)$ for some $(A,B)\in\sigma$, contradicting \cref{4tangleAvoidsClaw} as $A\cap B$ consists of three vertices.
\end{proof}

Recall that a separation \emph{distinguishes} two tangles if they orient it differently, and that it does so \emph{efficiently} if no separation of strictly lower order distinguishes the two tangles.

\begin{lemma}\label{CornerFish}
    Let $\tau_1,\tau_2$ be two 4-tangles in a 3-connected graph $G$ such that a 3-separation $\{A_1,A_2\}$ of $G$ efficiently distinguishes the $\tau_i$, say so that $\tau_i$ orients $\{A_{3-i},A_i\}$ towards~$A_i$ for both~$i$.
    Suppose that $\{C,D\}$ is a proper 3-separation of $G$ that crosses $\{A_1,A_2\}$, but such that both $\tau_i$ orient $\{C,D\}$ towards~$D$.
    Then there is $i$ such that the corner $\{C\cup A_{3-i},D\cap A_i\}$ efficiently distinguishes the~$\tau_i$.
\end{lemma}
\begin{proof}
    Let $c_i:=\{C\cup A_{3-i},D\cap A_i\}$ and $d_i:=\{C\cap A_{3-i},D\cup A_i\}$ for both~$i$, see \cref{fig:cornerfish}.
    
    We claim that there is an index $i$ such that $c_i$ has order $|c_i|\le 3$, and assume for a contradiction that both $c_1$ and $c_2$ have order greater than three.
    Recall that $|c_i|+|d_i|\le |A_1\cap A_2|+|C\cap D|=6$ for both~$i$ by submodularity~\cite[§12.5]{DiestelBookCurrent}.
    Then both $d_i$ have order at most two.
    As $G$ is 3-connected, this means that there are no vertices in the sets $(C\sm D)\cap (A_i\sm A_{3-i})$ for both~$i$.
    Since $\{C,D\}$ is proper, there is a vertex $v$ in $C\sm D$, and $v$ must lie in $A_1\cap A_2$.
    Now the three vertices in $C\cap D$ lie in the union of the separators of the~$d_i$.
    So the separator of $d_1$, say, contains at least two vertices from $C\cap D$.
    But this separator also contains~$v$, so it has size at least three, contradicting that we deduced above that it has size at most two.

    So $c_1$, say, has order at most three.
    By the tangle-property, $\tau_1$ must orient~$c_1$ towards the side $D\cap A_1$, while $\tau_2$ must orient $c_2$ towards the side $C\cup A_2$.
\end{proof}

\begin{figure}[ht]
    \centering
    \includegraphics[height=6\baselineskip]{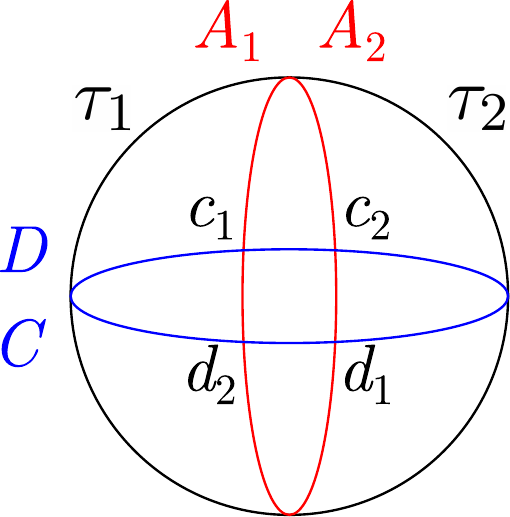}
    \caption{The situation in the proof of \cref{CornerFish}}
    \label{fig:cornerfish}
\end{figure}

\begin{lemma}\label{addingDistinguishers}
Let $N$ be a nested set of proper 3-separations of a 3-connected graph~$G$.
Let $\tau_1,\tau_2$ be two 4-tangles in~$G$ that include the same splitting star~$\sigma$ of~$N$.
Then there is a claw-free proper 3-separation of~$G$ that efficiently distinguishes the $\tau_i$ and interlaces~$\sigma$.
\end{lemma}
\begin{proof}
Let $\{A_1,A_2\}$ be a 3-separation of~$G$ that efficiently distinguishes the~$\tau_i$, chosen so that it crosses as few elements of~$\sigma$ as possible.
Then $\{A_1,A_2\}$ is proper and claw-free by \cref{4tangleAvoidsClaw}.

We claim that $\{A_1,A_2\}$ crosses no elements of~$\sigma$, and suppose for a contradiction that it crosses some $(C,D)\in\sigma$.
Both $\tau_i$ contain~$(C,D)$, but orient $\{A_1,A_2\}$ differently.
By \cref{CornerFish}, some corner $c=\{A_1\cap D,A_2\cup C\}$ (say) efficiently distinguishes the~$\tau_i$.
Every element of~$\sigma$ that crosses~$c$ must cross $\{A_1,A_2\}$ or~$\{C,D\}$ by \cite[Lemma~12.5.5]{DiestelBookCurrent}.
Since every element of~$\sigma$ is nested with~$\{C,D\}$, and since $c$ is nested with~$\{C,D\}$, the corner~$c$ crosses less elements of~$\sigma$ than~$\{A_1,A_2\}$, a contradiction.

Since $\{A_1,A_2\}$ is nested with all the elements of~$\sigma$, and since both $\tau_i$ include~$\sigma$, the separation $\{A_1,A_2\}$ must interlace~$\sigma$.
\end{proof}

\begin{corollary}
Every nested set of proper 3-separations of a 3-connected graph~$G$ can be extended to a nested set of proper 3-separations of~$G$ that efficiently distinguishes all 4-tangles in~$G$.\qed
\end{corollary}

\begin{lemma}\label{maxChopDists}
Every maximal 3-chop of a 3-connected graph~$G$ efficiently distinguishes all the 4-tangles in~$G$.
\end{lemma}
\begin{proof}
Let $N$ be a maximal 3-chop of~$G$.
Suppose for a contradiction that two distinct 4-tangles $\tau_1,\tau_2$ in~$G$ include the same splitting star~$\sigma$ of~$N$.
By \cref{addingDistinguishers}, there is a claw-free proper 3-separation $s$ of~$G$ that efficiently distinguishes the~$\tau_i$ and is nested with~$N$.
Then $N\cup\{s\}$ is a larger 3-chop of~$G$, a contradiction.
\end{proof}

\begin{theorem}\label{greedyTangleProperties}
Let $G$ be a 3-connected graph, and let $N$ be a maximal 3-chop of~$G$.
\begin{enumerate}
    \item\label{N:dist} $N$ efficiently distinguishes all the 4-tangles in~$G$.
    \item\label{N:tff} All torsos of~$N$ are faithful minors of~$G$.
\end{enumerate}
Let $\sigma$ be a splitting star of~$N$ with torso~$X$.
\begin{enumerate}[resume]
    \item\label{N:Four} If $|X|\le 4$, then:
    \begin{itemize}
        \item some 4-tangle in~$G$ includes $\sigma$ if and only if $V(X)$ is cubic in~$G$;
        \item if a 4-tangle $\tau$ in~$G$ includes $\sigma$ and $\varphi\colon G\rminor Q$ is a standard cube-minor at~$V(X)$, then $\varphi(\tau)$ is the unique 4-tangle in~$Q$.
    \end{itemize}
    \item\label{N:Five} If $|X|\ge 5$ and $G\neq K_{3,3}$, then:
    \begin{itemize}
        \item the torso $X$ is internally 4-connected;
        \item a 4-tangle $\tau$ in~$G$ includes~$\sigma$ and $\varphi(\tau)$ is the unique 4-tangle in~$X$ for all $\varphi\colon G\rminor X$ witnessing that $X$ is a faithful minor of~$G$.
    \end{itemize}
\end{enumerate}
\end{theorem}
\begin{proof}
\cref{N:dist} is \cref{maxChopDists}.
\cref{N:tff} is \cref{chopExtension}.

\cref{N:Four}. 
If $|X|\le 3$, then every 4-tangle in~$G$ lives in some component of $G\sm X$ in the sense of \cref{TangleDirection}, and by the same lemma no 4-tangle in~$G$ includes~$\sigma$.
Hence we may assume that $|X|=4$.
\cref{cubicTanglesInFourSets} shows the first claim.
Suppose now that a 4-tangle $\tau$ in~$G$ includes~$\sigma$ and $\varphi\colon G\rminor Q$ is a standard cube-minor at~$V(X)$.
Since the cube~$Q$ is internally 4-connected, it has a unique 4-tangle~$\theta_Q$, by \cref{internalUnique}.
Let $\theta$ denote the $\varphi$-lift of~$\theta_Q$ to~$G$.
Then $\theta$ includes~$\sigma$, by \cref{cubicTanglesInFourSets}.
By~\cref{N:dist}, this implies $\theta=\tau$, and hence $\varphi(\tau)=\theta_Q$.

\cref{N:Five}.
Assume $|X|\ge 5$.
By \cref{largeTorsoInt4con} and since $G$ is not a~$K_{3,3}$, the torso~$X$ is internally 4-connected.
Hence $X$ has a unique 4-tangle~$\tau_X$ by \cref{internalUnique}.
Let $\varphi\colon G\rminor X$ witness that $X$ is a faithful minor of~$G$.
Then the $\varphi$-lift $\tau$ of~$\tau_X$ to~$G$  includes~$\sigma$ by \cref{liftIncludesStar}.
In particular, $\varphi(\tau)=\tau_X$.
By~\cref{N:dist}, no other 4-tangle besides $\tau$ includes~$\sigma$.
Hence $\varphi'(\tau)=\tau_X$ for all faithful minor-maps $\varphi'\colon G\rminor X$.
\end{proof}

% \begin{corollary}\label{cor:tw3}
% A graph has tree-width $\le 3$ if and only if it has no internally 4-connected minor.\qed
% \end{corollary}

% \begin{corollary}
% A planar 3-connected graph has path-width $\le 3$ if and only if it has no internally 4-connected minor.
% \end{corollary}
% \begin{proof}
% This follows from \cref{cor:tw3} as a 3-connected graph has path-width $\le 3$ if and only if it has tree-width $\le 3$ and no $K_{3,3}$-minor, and because the cube has tree-width~$>3$.
% \end{proof}

A graph is \emph{quasi 4-connected} if it is 3-connected, has more than four vertices, and every 3-separation of~$G$ has a side of size~$\le 4$.

\begin{corollary}\label{quasi4star}
Every quasi 4-connected graph~$G$ has a star-decomposition of adhesion three such that the central torso is internally 4-connected or~$K_4$ or~$K_3$, and all leaf-bags have size four.
\end{corollary}
\begin{proof}
If $G$ has no proper 3-separation, then the trivial star-decomposition is as desired.
If $G$ has a proper 3-separation with two sides of equal size, then the $K_{1,1}$-decomposition of $G$ into the two sides is as desired (it does not matter which bag is the central one).
So suppose that $G$ has a proper 3-separation and that every proper 3-separation of~$G$ has a largest side.
Let $N$ be a maximal 3-chop of~$G$.
Let $\sigma$ consist of all $(A,B)$ with $\{A,B\}\in N$ and $|A|=4$.
As all elements of~$N$ are proper and have a unique side of size four, $\sigma$ must be a star.
The torso of~$\sigma$ is a $K_3$ or $K_4$ or internally 4-connected by \cref{greedyTangleProperties}.
\end{proof}

\begin{figure}[ht]
    \centering
    \includegraphics[height=4.5\baselineskip]{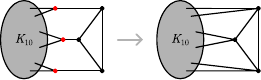}
    \caption{The converse of \cref{quasi4star} fails}
    \label{fig:quasiV4con}
\end{figure}

\begin{example}
The converse of \cref{quasi4star} fails in the following sense.
We claim that the graph $G$ depicted on the left-hand side of \cref{fig:quasiV4con} is 3-connected, has a star-decomposition of adhesion three such that the central torso is (internally) 4-connected, and all leaf-bags have size four -- but $G$ is not quasi 4-connected.
\end{example}
\begin{proof}
Let $N$ denote the set of all 3-separations whose separators are the neighbourhoods of the red vertices, and consider the star-decomposition defined by~$N$.
\end{proof}

\begin{lemma}\label{reduceTo3con}
For every $4$-tangle $\tau$ in a graph~$G$ there is $\varphi\colon G\rminor H$ such that $H$ is 3-connected and $\varphi(\tau)$ is a $4$-tangle in~$H$.
\end{lemma}
\begin{proof}
Let $\tau$ be a $4$-tangle in~$G$.
Then $\tau$ is the lift of a $4$-tangle in some block of~$G$, so we may assume without loss of generality that $G$ is 2-connected.
Let $N$ be a maximal nested set of proper 2-separations of~$G$.
Then every torso of~$N$ is either a complete graph on at most three vertices or 3-connected, by \cref{LiftingLemma}.
If the bag of a splitting star~$\sigma$ of~$N$ has size at most three, then $\sigma$ is not included in any $4$-tangle in~$G$.
Hence $\tau$ includes a splitting star of~$N$ whose torso~$H$ is 3-connected.
Then $\varphi(\tau)$ is a $4$-tangle in~$H$ for every faithful $\varphi\colon G\rminor H$.
\end{proof}

\begin{proof}[Proof of \cref{main}]
By \cref{internalUnique}, every internally 4-connected graph~$H$ has a unique 4-tangle~$\tau_H$.
It remains to show that, if $\theta$ is a 4-tangle in a graph~$G$, then there is a map $\varphi\colon G\rminor H$ to some internally 4-connected minor $H$ of~$G$ such that $\varphi(\theta)=\tau_H$.
Let $G$ and $\theta$ be given.
By \cref{reduceTo3con}, we may assume that $G$ is 3-connected.
As $K_{3,3}$ has no 4-tangle, we may assume that $G\neq K_{3,3}$.
Let $N$ be a maximal 3-chop of~$G$.
Then $\theta$ includes a unique splitting star~$\sigma$ of~$N$ with torso~$X$.
The result follows by \cref{greedyTangleProperties} \cref{N:Four} and \cref{N:Five}.
\end{proof}

\section{The large torsos of maximal 3-chops are unique up to isomorphism}\label{sec:Universality}

A 4-tangle $\tau$ in a 3-connected graph $G$ is \emph{cubic} if it includes a star of 3-separations of $G$ whose bag is a cubic vertex set in~$G$.
Note that $\tau$ is cubic if and only if there exist a cubic vertex set $X\se V(G)$ and a standard cube-minor $\varphi\colon G\rminor Q$ at~$X$ such that $\varphi(\tau)$ is the unique 4-tangle in~$Q$, by \cref{cubicTanglesInFourSets}.

A class $\cC$ of 3-chops of 3-connected graphs \emph{endorses 4-tangles} if for every non-cubic 4-tangle $\tau$ in a 3-connected graph~$G$ and every two $N_1,N_2\in\cC$ the splitting stars $\sigma_i$ of $N_i$ included in~$\tau$ (for $i=1,2$) have isomorphic internally 4-connected torsos.

The main result of this section is:

\begin{theorem}\label{universality}
The class of maximal 3-chops of 3-connected graphs endorses 4-tangles.
\end{theorem}

\begin{example}
In the definition of `endorses 4-tangles', we do no require that cubic 4-tangles always include $\sigma_i$ whose bags are cubic vertex sets, because this is not true for maximal 3-chops.
To see this, let $Q$ be a cube with bipartition classes $A$ and~$B$.
Let $G$ be obtained from $Q$ by replacing an arbitrary number of vertices $a\in A$ with large cliques and joining them completely to all neighbours of~$a$ in~$B$.
Then $G$ has a unique cubic 4-tangle~$\tau$, which includes a splitting star $\sigma$ of every maximal 3-chop of~$G$.
The bag of~$\sigma$ can contain anywhere between four and eight vertices.
\end{example}

To prove \cref{universality}, we use some machinery from~\cite{Canonical3arXiv}.
For this, we assume familiarity with a few terms from~\cite{Canonical3arXiv}, such as \lq tri-separations\rq,  their \lq reductions\rq, and \lq totally nested\rq.

\subsection{Proof overview}

Consider an arbitrary 3-connected graph~$G$ with a non-cubic 4-tangle~$\tau$.
Let $N_1,N_2$ be maximal 3-chops of~$G$ with splitting stars~$\sigma_1,\sigma_2\se\tau$ respectively.
If the entire graph $G$ is internally 4-connected, then $N_1=\emptyset=N_2$, so we are done immediately.
Hence by a result from~\cite{Canonical3arXiv}, we may assume that $G$ has a totally-nested tri-separation~$\{C,D\}$.
The 4-tangle $\tau$ will live either in $C$ or in~$D$, say in~$D$ (\cref{tauOrientsTrisep}).

If the separator of $\{C,D\}$ consists of three vertices, then we can show that $\{C,D\}$ is nested with all separations in $N_1$ and $N_2$, so $\{C,D\}$ lies in both $N_1$ and $N_2$ by maximality of the 3-chops.
In this case, the splitting stars~$\sigma_i$ live in the side~$D\subsetneq V(G)$, so we are essentially done by induction.

Otherwise, the separator of $\{C,D\}$ contains an edge~$e$ with endvertices $c\in C\sm D$ and $d\in D\sm C$.
Here we plan to apply induction to~$G/e$.
For this, we have to show that $e$ has at least one endvertex $v_i$ outside of the torso of~$\sigma_i$, to preserve internal 4-connectivity of the torso.
The endvertex $v_i$ may depend on~$i$.
To find~$v_i$, we obtain a 3-separation $\{C_i,D_i\}$ from $\{C,D\}$ that is nested with all separations in~$N_i$ (\cref{triSepExtendsChop}).
The separator of $\{C_i,D_i\}$ consists of the vertices in $C\cap D$ plus a choice of an endvertex of each edge in the separator of $\{C,D\}$.
For the edge $e$, we try to choose its endvertex $d$ to be in $C_i\cap D_i$.
If successful, we may then take $v_i:=c\in C_i\sm D_i$, as it is not hard to show that $\sigma_i$ lives in~$D_i$.
Otherwise there is $\{X_i,Y_i\}\in N$ that witnesses why we could not choose~$d$: the vertex $c$ lies in $X_i\cap Y_i$ but $d$ lies in $Y_i\sm X_i$, say.
Moreover, $d$ is the unique neighbour of $c$ in $Y_i\sm X_i$.
Here we argue that $\sigma_i$ lives in~$X_i$ (\cref{v-path} and \cref{claim7}), which allows us to take $v_i:=d$.
This will complete the proof.

Using the totally-nested tri-separation $\{C,D\}$ has the following advantage: The edge $e$ in the above overview must be chosen independently of~$i$, yet $e$ must have an endvertex outside the torso of each~$\sigma_i$, and $\{C,D\}$ readily provides such an edge~$e$ (or we are done almost immediately).
We have explored proofs of \cref{universality} without tri-separations, but they all eventually led us to considerations that are reminiscent of tri-separations.

\subsection{Proof of \texorpdfstring{\cref{universality}}{Theorem 4.1}}

Let $(C,D)$ be a non-trivial tri-separation of a 3-connected graph~$G$.
The \emph{right-shift} of $(C,D)$ is the 3-separation $(\hat C,D)$ where $\hat C$ is obtained from $C$ by adding every endvertex in $D\sm C$ of every edge in the separator of~$(C,D)$.
The \emph{left-shift} $(C,\hat D)$ is defined similarly, with the roles of $C$ and $D$ reversed.

\begin{lemma}\label{tauOrientsTrisep}
    Let $G$ be a 3-connected graph with a tri-separation $(C,D)$ and a 4-tangle~$\tau$.
    Let $(C,\hat D)$ and $(\hat C,D)$ denote the left-shift and the right-shift, respectively.
    Then $\tau$ lives either in $C$ or in $D$ in the following sense: either $(\hat D,C)\in\tau$ or $(\hat C,D)\in\tau$.
\end{lemma}
\begin{proof}
    Assume that $(C,\hat D)\in\tau$, say.
    If the separator of $(C,D)$ consists of vertices only, this means that $(C,\hat D)=(\hat C,D)\in\tau$ and we are done.
    So write $(C_0,D_0):=(C,\hat D)$ and assume that $cd$ is an edge in the separator of $(C,D)$ with $c\in C_0\sm D_0$ and $d\in D_0\sm C_0$.
    Let $C_1:=C_0\cup \{d\}$ and $D_1:=D_0\sm\{c\}$, so the separator of $(C_1,D_1)$ is obtained from the separator of $(C_0,D_0)$ by replacing $c$ with~$d$.
    Then $G=G[C_0]\cup G[D_1]\cup G[\{c,d\}]$.
    We apply the tangle-property to the following three separations: $(C_0,D_0)$, $(D_1,C_1)$ and $(\{c,d\},V(G))$.
    Since $(C_0,D_0)=(C,\hat D)\in\tau$ by assumption and $(\{c,d\},V(G))\in\tau$ as $\tau$ is a tangle, this yields $(C_1,D_1)\in\tau$.
    Proceeding in this manner for any other edges in the separator of $(C,D)$, we eventually find that $(C_k,D_k)=(\hat C,D)\in\tau$.
\end{proof}

\begin{lemma}\label{trisepVS3sep}
    Let $G$ be a 3-connected graph with a tri-separation $(C,D)$.
    Let $(A,B)$ be an oriented proper 3-separation of~$G$, and let $(\bar A,\bar B)$ be the reduction of $(A,B)$.
    Assume that $(\bar A,\bar B)\le (C,D)$.
    Then the right-shift $(\hat C,D)$ of $(C,D)$ satisfies $(A,B)\le (\hat C,D)$.
\end{lemma}
\begin{proof}
    Since $B\supseteq \bar B\supseteq D$, it remains to show that $A\se \hat C$.
    For this, assume for a contradiction that there is a vertex $v\in A\sm \hat C$.
    Since $\bar A\se C\se\hat C$, this means that $v\in A\cap B$ was reduced to an edge $vw$ in the separator of $(\bar A,\bar B)$, and $w\in \bar A\sm \bar B\se C$.
    The vertex $w$ cannot lie in $C\sm D$ as $v\notin\hat C$.
    So $w$ must lie in $C\cap D$.
    Hence $w\in D\sm \bar B$, which contradicts that we have $\bar B\supseteq D$.
\end{proof}

\begin{corollary}
    Let $G$ be a 3-connected graph with a totally-nested non-trivial tri-separation $(C,D)$ such that the separator of $(C,D)$ consists of three vertices.
    Then $(C,D)$ is nested with every proper 3-separation of~$G$.
\end{corollary}
\begin{proof}
     Let $(\bar A,\bar B)$ be the reduction of $(A,B)$.
     Since $(\bar A,\bar B)$ is a tri-separation, it is nested with the totally nested tri-separation $(C,D)$; say $(\bar A, \bar B)\le (C,D)$.
     Then $(A,B)\le (C,D)$ by \cref{trisepVS3sep}.
\end{proof}

\begin{lemma}\label{triSepExtendsChop}
    Let $G$ be a 3-connected graph, and let $N$ be a nested set of proper 3-separations of~$G$.
    Let $(C,D)$ be a totally-nested non-trivial tri-separation of~$G$.
    Then there is a 3-separation $\{C',D'\}$ of $G$ that is nested with all separations in~$N$ such that $C\se C'$ and $D\se D'$.
    Moreover, $\{C',D'\}$ can be chosen so that for each edge $cd$ in the separator of $(C,D)$ with $c\in C\sm D$ and $d\in D\sm C$, we either have $c\in C'\sm D'$ and $d\in C'\cap D'$ or some separation in $N$ has an orientation $(A,B)$ such that $c\in A\cap B$ and $d$ is the unique neighbour of $c$ in $B\sm A$.
\end{lemma}
\begin{proof}
    Write $\{\,\{A_i,B_i\}:i\in I\,\}:=N$.
    For each $i\in I$, denote the reduction of $(A_i,B_i)$ by $(\bar A_i,\bar B_i)$.
    Since $(C,D)$ is totally-nested, it is nested with every $(\bar A_i,\bar B_i)$. 
    Without loss of generality, the sides $A_i$ and $B_i$ are named so that we have $(\bar A_i,\bar B_i)\le (C,D)$ or $(C,D)\le (\bar A_i,\bar B_i)$ for all~$i\in I$.
    Let $\gamma$ be the set of all $i\in I$ with $(\bar A_i,\bar B_i)\le (C,D)$.
    Similarly, let $\delta$ be the set of all $i\in I$ with $(C,D)\le (\bar A_i,\bar B_i)$.
    Let $C'':=C\cup\bigcup_{i\in\gamma} A_i$ and $D'':=D\cup\bigcup_{j\in\delta} B_j$.
    Then $\{C'',D''\}$ is a mixed-separation of $G$ which satisfies $C\se C''$ and $D\se D''$.

    \begin{claim}
    $\{C'',D''\}$ has order~3.
    \end{claim}
    \begin{cproof}
    Let $(\hat C,D)$ and $(\hat D,C)$ denote the right-shift and left-shift of $(C,D)$, respectively.
    We have $C\se C''\se \hat C$ and $D\se D''\se\hat D$ by \cref{trisepVS3sep}.
    Hence 
    \begin{equation}\label{eq:shiftSqueeze}
        (C,\hat D)\le (C'',D'')\le (\hat C,D).
    \end{equation}    
    So it suffices to show that for every edge $cd$ in the separator of $\{C,D\}$ with $c\in C\sm D$ and $d\in D\sm C$, not both $c$ and $d$ are in the separator of $\{C'',D''\}$.
    Assume for a contradiction that $\{c,d\}\se C''\cap D''$.
    Then there are indices $i\in\gamma$ and $j\in\delta$ such that $d\in A_i$ and $c\in B_j$.
    Since $d\in A_i\sm C$ but $\bar A_i\se C$, the vertex $d$ must have been reduced to an edge in the separator of $(\bar A_i,\bar B_i)$, and this edge can only be $cd$.
    Hence $c\in A_i\sm B_i$.
    However, $(\bar A_i,\bar B_i)\le (C,D)\le (\bar A_j,\bar B_j)$ implies $(A_i,B_i)\le (A_j,B_j)$ as $\{A_i,B_i\}$ and $\{A_j,B_j\}$ are nested by assumption.
    So $c\in B_j\se B_i$ contradicts $c\in A_i\sm B_i$.
    \end{cproof}\medskip

    Next, we show that $\{C'',D''\}$ is nested with $\{A_i,B_i\}$ for all~$i\in I$.
    By symmetry we may assume $i\in\gamma$.
    Then $A_i\se C''$ by definition of~$C''$, on the one hand.
    On the other hand, $B_i\supseteq D\cup\bigcup_{j\in\delta}B_j$ since $(A_i,B_i)\le (A_j,B_j)$ for all $j\in\delta$, as above.
    The union on the right-hand side equals~$D''$.
    Hence $\{C'',D''\}$ is nested with all~$\{A_i,B_i\}$.

    We turn $\{C'',D''\}$ into a 3-separation $\{C',D'\}$ of~$G$ by taking $D':=D''$ and obtaining $C'$ from $C''$ by adding the endvertex in $D''\sm C''$ of every edge in the separator of~$\{C'',D''\}$.

    \begin{claim}\label{theException}
        Let $cd$ be an edge in the separator of $(C,D)$ with $c\in C\sm D$ and $d\in D\sm C$.
        Assume that either $c\notin C'\sm D'$ or $d\notin C'\cap D'$.
        Then there is $i\in I$ such that $c\in A_i\cap B_i$ and $d$ is the unique neighbour of $c$ in~$B_i\sm A_i$.
    \end{claim}
    \begin{cproof}
        The assumption implies that $c$ is contained in the separator of $\{C'',D''\}$.
        So there is $i\in\delta\se I$ such that $c\in B_i$.
        Since $C\se\bar A_i\se A_i$, we have $c\in A_i\cap B_i$.
        As $G$ is 3-connected, $c$~has a neighbour in $B_i\sm A_i\se \hat D\sm C=D\sm C$ (using \cref{trisepVS3sep}).
        The vertex $d$ is the unique neighbour of $c$ in $D\sm C$, so it also is the unique neighbour of $c$ in $B_i\sm A_i$.
    \end{cproof}\medskip

    The `moreover'-part of the lemma holds by \cref{theException}.
\end{proof}

\begin{lemma}\label{v-path}
Let $G$ be a 3-connected graph, and let $N$ be a nested set of proper 3-separations of~$G$. 
Assume that for some $\{C,D\}\in N$, there is a vertex $v\in C\cap D$ that has a unique neighbour $w$ in $D\sm C$.
Let $O$ denote the set of all 3-separations $(A,B)$ with $\{A,B\}\in N$ and $(C,D)\le (A,B)$ and $v\in A\cap B$.
Then $w\in B\sm A$ for all $(A,B)\in O$, and $\le$ linearly orders~$O$.
\end{lemma}

\begin{proof}
Let $(A,B)\in O$. Since $G$ is 3-connected, the vertex $v\in A\cap B$ has a neighbour in $B\sm A$.
Since $(C,D)\leq (A,B)$ we have that $B\sm A\se D\sm C$. 
Since $v$ only has the neighbour $w$ in $D\sm C$, we conclude that $w\in B\sm A$.

Now to show that $\le$ linearly orders~$O$, let $(A,B),(A',B')\in O$.
Since both $\{A,B\}$ and $\{A',B'\}$ are contained in the nested set~$N$, they have orientations that can be compared by~$\le$.
Since $(C,D)\le (A,B)$ and $(C,D)\le (A',B')$, it must hold, after possibly interchanging the names of $(A,B)$ and $(A',B')$, that either $(A',B')\le (A,B)$ or $(B',A')\le (A,B)$.
But the latter case contradicts $w\in (B'\sm A')\cap (B\sm A)$.
\end{proof}

Given a tangle $\tau$ and an edge $e=vw$, we say that a separation $\{A,B\}$ \emph{separates} $\tau$ from $e$ if $\tau$ lives in $A$ and one of the endvertices $v$ or $w$ of $e$ is in $B\sm A$; or the same with the roles of \lq $A$\rq\ and \lq $B$\rq\ interchanged.

\begin{lemma}\label{claim7}
Let $G$ be a 3-connected graph, and let $N$ be a nested set of proper 3-separations of~$G$.
Let $\tau$ be a 4-tangle in~$G$, and let $\sigma$ be the splitting star of $N$ with $\sigma\se\tau$.
Assume that the torso $H$ of $\sigma$ is internally $4$-connected.
Assume that there is $\{C,D\}\in N$ with a vertex $v\in C\cap D$ such that $v$ has a unique neighbour $w$ in $D\sm C$.
Then some $\{A,B\}\in N$ separates $\tau$ from $e:=vw$.
\end{lemma}

\begin{proof}
Let $O$ be defined as in the statement of \autoref{v-path}. 
By that lemma, $\le$ linearly orders~$O$.
So we may write the elements of $O$ as
\[
    (A_0,B_0) < (A_1,B_1) < \cdots < (A_n,B_n)
\]
with $w\in B_i\sm A_i$ for all $i\le n$.
Note that $(A_0,B_0)=(C,D)$.

\begin{claim}\label{not-the-end}
There is $i\le n$ such that $\tau$ orients $\{A_i,B_i\}$ towards $A_i$, or some $\{X,Y\}\in N$ separates $\tau$ from $e=vw$.
\end{claim}
\begin{cproof}
Suppose that $\tau$ orients $\{A_i,B_i\}$ towards $B_i$ for all~$i\le n$.
As $\sigma\se\tau$ and $\tau$ orients $\{A_n,B_n\}\in N$ towards $B_n$, there is $(X,Y)\in\sigma$ with $(A_n,B_n)\le (X,Y)$.
In particular, $\tau$ orients $\{X,Y\}$ towards~$Y$.

Assume first that $(A_n,B_n)<(X,Y)$.
So $(X,Y)\notin O$, which combined with $(C,D)\le (X,Y)$ means that $v$ is contained in $X\sm Y$.
Thus $\{X,Y\}$ separates $\tau$ from $e=vw$.

So we may assume that $(A_n,B_n)=(X,Y)$.
We will show that this case is impossible.
Since $(A_n,B_n)$ is the $\le$-maximal element of~$O$, it also is the only element of $\sigma$ in~$O$.
By the definition of a star, every $(U,W)\in\sigma$ distinct from $(A_n,B_n)$ satisfies $(C,D)\le (A_n,B_n)<(W,U)$.
As $(W,U)\notin O$, the vertex $v$ cannot be contained in $U\cap W$; thus $v\in W\sm U$ and so $w\in W$.
Therefore, $(A_n,B_n)$ is the only element of $\sigma$ that contains $v$ in its separator.
As $w$ is the unique neighbour of $v$ in $B_n\sm A_n$, it follows that $v$ has exactly three neighbours in the torso $H$ of $\sigma$: the vertex $w$ and the two vertices $s,t$ in the separator $A_n\cap B_n$ besides~$v$.
As $A_n\cap B_n$ induces a triangle in~$H$, the two vertices $s,t$ form an edge in~$H$.
Then $v$ is a vertex of $H$ with degree three but such that two of its neighbours are joined by an edge, contradicting that $H$ is internally $4$-connected.
Hence $(A_n,B_n)=(X,Y)$ is impossible.
\end{cproof}\medskip

If there is $i\le n$ such that $\tau$ orients $\{A_i,B_i\}$ towards~$A_i$, then $\{A_i,B_i\}$ separates $\tau$ from $e=vw$ as $w\in B_i\sm A_i$.
Otherwise some $\{X,Y\}\in N$ separates $\tau$ from $e$ by \cref{not-the-end}.
\end{proof}

\begin{proof}[Proof of \cref{universality}]
Let $G$ be a 3-connected graph with a non-cubic 4-tangle~$\tau$.
Let $N_1,N_2$ be maximal 3-chops of~$G$.
For each $i=1,2$ let $\sigma_i$ denote the splitting star of~$N_i$ included in~$\tau$.
Let $H_i$ denote the torso of~$\sigma_i$.
Then both $H_i$ are internally 4-connected by \cref{greedyTangleProperties}.
We have to find an isomorphism $H_1\to H_2$.

We proceed by induction on the number of vertices of~$G$.
The induction starts with the case that $G$ is internally 4-connected, as here both $N_i$ are empty so that $H_1=G=H_2$.

For the induction step, assume now that $G$ is not internally 4-connected.
By the angry tri-separation theorem \cite[Theorem~1.1.5]{Canonical3arXiv}, either $G$ has a totally-nested non-trivial tri-separation, or $G$ is a wheel, or $G$ is a $K_{3,m}$ for some $m\ge 3$.
As wheels and $K_{3,m}$'s have no 4-tangles, we find that $G$ has a totally-nested non-trivial tri-separation~$(C,D)$.
We consider two cases.

\underline{In the first case}, the tri-separation $(C,D)$ only has vertices in its separator, so $\{C,D\}$ is a proper 3-separation.
By \cref{triSepExtendsChop}, $\{C,D\}$ is nested with all separations in both~$N_i$.
Since $\{C,D\}$ is non-trivial, it is claw-free.
So $\{C,D\}\in N_i$ for both~$i$ by maximality.
Without loss of generality, $(C,D)\in\tau$.
Let $G^D$ denote the torso of $\{(C,D)\}$ with vertex-set~$D$.
Each $N_i$ projects to a maximal 3-chop $N^D_i$ of $G^D$ which includes the projection $\sigma^D_i$ of~$\sigma_i$.
Hence the torso $H^D_i$ of $\sigma^D_i$ equals~$H_i$ and, in particular, it is internally 4-connected.
So each $H^D_i$ has a unique 4-tangle $\tau^D_i$ by \cref{internalUnique}.
Since $H^D_i=H_i$, its unique 4-tangle $\tau^D_i$ lifts to~$\tau$ in~$G$.
Hence both $\tau^D_i$ lift to the same 4-tangle $\tau^D$ in~$G^D$.
The 4-tangle $\tau^D$ is non-cubic, as any star of 3-separations of $G^D$ that says otherwise lifts to a star of 3-separations of $G$ showing that $\tau$ is cubic after all.
Applying the induction hypothesis to the maximal 3-chops $N^D_i$ of $G^D$ and the non-cubic 4-tangle~$\tau^D$ yields an isomorphism $H^D_1\to H^D_2$.
Since $H^D_i=H_i$ this also is an isomorphism $H_1\to H_2$.

\underline{In the second case}, the separator of $(C,D)$ contains an edge $e=cd$, named so that $c\in C\sm D$ and $d\in D\sm C$.
By \cref{tauOrientsTrisep}, we may assume that $(\hat C,D)\in\tau$, where $(\hat C,D)$ denotes the right-shift of~$(C,D)$.
By \cref{triSepExtendsChop}, for both~$i$ there is a 3-separation $(C_i,D_i)$ of $G$ that is nested with all separations in~$N_i$ such that $C\se C_i$ and $D\se D_i$.
Moreover, the lemma allows us to choose $(C_i,D_i)$ so that either $c\in C_i\sm D_i$ and $d\in C_i\cap D_i$, or some separation in $N_i$ has an orientation $(A_i,B_i)$ such that $c\in A_i\cap B_i$ and $d$ is the unique neighbour of $c$ in $B_i\sm A_i$.

\begin{claim}
Each $N_i$ has an element $\{X_i,Y_i\}$ which $\tau$ orients towards~$Y_i$ while at least one of $c,d$ is contained in~$X_i\sm Y_i$.
\end{claim}
\begin{cproof}
If $c\in C_i\sm D_i$ and $d\in C_i\cap D_i$, then $\{C_i,D_i\}$ is proper and claw-free as $(C,D)$ is non-trivial and as $\tau$ orienting $\{\hat C,D\}$ towards~$D$ implies that $G[D]$ contains a cycle by \cref{4tangleAvoidsClaw}.
Then, as $\{C_i,D_i\}$ is nested with all separations in~$N_i$, we get that $\{C_i,D_i\}$ is contained in~$N_i$ by maximality, so we may take $(X_i,Y_i):=(C_i,D_i)$.
Otherwise there is $\{A_i,B_i\}\in N_i$ such that $c\in A_i\cap B_i$ and $d$ is the unique neighbour of $c$ in $B_i\sm A_i$.
Then \cref{claim7} produces a suitable~$\{X_i,Y_i\}$.
\end{cproof}\medskip

Finally, we perform the induction.
The graph $G':=G/e$ is 3-connected as $e$ stems from a tri-separator.
Let $M_i$ consist of all $\{A,B\}\in N_i$ such that $(X_i,Y_i)\le (A,B)$ or $(X_i,Y_i)\le (B,A)$.
Note that $\sigma_i$ is a splitting star of~$M_i$.
Let $M'_i$ be the projection of $M_i$ to $G/e$, which is a 3-chop of $G/e$, but it is not clear a priori that it is a maximal one.
Let $N'_i$ be an arbitrary extension of $M'_i$ to a maximal 3-chop of~$G/e$.
Since at most one of the endvertices $c,d$ of $e$ is contained in $Y_i$,
the splitting star $\sigma_i$ carries over to a splitting star~$\sigma'_i$ of~$N'_i$ so that the torso $H_i$ of $N_i$ equals the torso $H'_i$ of $\sigma'_i$ (up to possibly renaming $c$ or~$d$).
In particular, we can argue similarly to the first case that the 4-tangle $\tau'_i$ of $H'_i$ lifts to a 4-tangle $\tau'$ of $G'$ which is the same for both~$i$.
Moreover, $\tau'$ is not cubic, also by a similar argument.
We may therefore apply the induction hypothesis to the maximal 3-chops $N'_i$ of $G'$ and the non-cubic 4-tangle~$\tau'$ to find an isomorphism $H'_1\to H'_2$.
Each torso $H_i$ is isomorphic to $H'_i$, and we combine all three isomorphisms to an isomorphism $H_1\to H_2$.
\end{proof}

\section{A proof of Kuratowski's Theorem via internal 4-connectedness}\label{sec:Kuratowski}

\begin{lemma}\label{FaceBoundaryCycle}
\cite[Prop.~4.2.6]{DiestelBookCurrent}
Every face of a 2-connected plane graph is bounded by a cycle.
\end{lemma}

\begin{lemma}\label{cycleUV}
Let $G$ be an excluded minor for the class of planar graphs.
If $G$ is internally 4-connected, then $G-u-v$ is a cycle for every edge $uv$ of~$G$.
\end{lemma}
\begin{proof}
Suppose that $G$ is internally 4-connected, and let $uv$ be an arbitrary edge of~$G$.
By assumption, the graph $G/e$ is planar, so the graph $G':=G-u-v$ has a drawing in which the neighbourhood $N(u,v)$ lies on the boundary of a single face~$f$.
Since $G$ is internally 4-connected, the graph $G'$ is 2-connected.
Hence by \cref{FaceBoundaryCycle}, the face $f$ is bounded by a cycle~$O$.
Let $H$ denote the subgraph of $G$ that is obtained from~$O$ by adding both vertices $u,v$ and all the edges in $G$ that are incident with~$u,v$.
Note that $G$ is planar if and only if $H$ is planar.
Hence $H$ is not planar.
By minimality of~$G$, we have $G=H$.
\end{proof}

\begin{lemma}\label{internalK5}
If an excluded minor for the class of planar graphs is internally 4-connected, then it must be isomorphic to~$K_5$.
\end{lemma}
\begin{proof}
Let $G$ be an excluded minor for planarity such that $G$ is internally 4-connected.
Let $uv$ be an edge of~$G$.
By \cref{cycleUV}, $O:=G-u-v$ is a cycle.

We claim that $O$ has length at most four.
Assume for a contradiction that $O$ has length at least five, and let $xy$ be an arbitrary edge on~$O$.
Then the path $O-x-y$ has at least one internal vertex, and we let $w$ be an arbitrary such vertex.
By \cref{cycleUV}, the graph $G-x-y$ is a cycle.
Hence the internal vertex $w$ of the path $O-x-y$ is not adjacent to $u$ or~$v$.
But then the graph obtained from $G$ by contracting an edge incident to the degree-two vertex~$w$ is not planar (as $G$ is not planar), contradicting that $G$ is an excluded minor for planarity.

So assume that $O$ has length at most four.
We apply \cref{cycleUV} to all edges of~$O$ in some linear order once around~$O$.
If $O$ has length four, we find that the neighbours in~$G$ of $u$ and of~$v$ on~$O$ form two disjoint independent 2-sets in~$O$.
But then $G$ is a $K_{3,3}$, contradicting that $G$ is internally 4-connected.
Hence $O$ has length three, in which case we find that $G$ is a~$K_5$.
\end{proof}

\begin{proof}
[Proof of Kuratowski's Theorem]
Containing $K_5$ or $K_{3,3}$ as a minor obstructs planarity by the usual arguments via Euler's formula.
For the hard implication, let $G$ be an excluded minor for planarity.
It is easy to see that $G$ is 2-connected: otherwise all blocks of $G$ are proper subgraphs of~$G$, and hence planar, so we obtain a drawing of~$G$ by combining the drawings of its blocks, a contradiction.
Similarly, we can show that $G$ is 3-connected: otherwise $G$ has a 2-separator $X$ that leaves components $C_i$, and each graph $H_i$ that is obtained from $G[C_i\cup X]$ by turning $X$ into a $K_2$ is a planar minor of~$G$, so the drawings of the $H_i$ combine to a drawing of~$G$, a contradiction.
Hence we may assume that $G$ is 3-connected.

If $G$ is internally 4-connected, then $G$ is a~$K_5$ by \cref{internalK5}.
Otherwise $G$ is a $K_{3,3}$ or $G$ has a claw-free 3-separation $\{A,B\}$.
In the latter case, let $G_A$ be obtained from $G[A]$ by turning the separator $X:=A\cap B$ into a triangle, and note that $G_A$ is a minor of $G$ by \cref{3sepClawChar}~\cref{C:twoCycles} and Menger's theorem.
Similarly, $G_B$ is a minor of~$G$.
Hence $G_A$ and $G_B$ have planar drawings.
These drawings can be combined into a drawing of~$G$ if both $G_A\sm X$ and $G_B\sm X$ are connected.
Otherwise $G$ contains $K_{3,3}$ as a minor, and must be equal to $K_{3,3}$ by minimality.
\end{proof}

\begin{ack}
    We thank the referee for carefully reading the paper, spotting a mistake in the proof of \cref{universality} and suggesting a shorter proof of \cref{v-path}.
\end{ack}

\bibliographystyle{amsplain}
\bibliography{literatur}
\end{document}